\documentclass[12pt]{article}
\usepackage{amsfonts}
\usepackage{amsmath}
\usepackage{amssymb}
\usepackage{amsthm}
\usepackage{geometry}
\usepackage{rotating}
\usepackage[utf8]{inputenc}
\usepackage{array}
\usepackage[english]{babel}
\usepackage{float}

\usepackage{bm}
\usepackage{natbib}

\def\P{\mbox{P}}
\theoremstyle{plain} \newtheorem{theorem}{Theorem}
\theoremstyle{plain} \newtheorem{corollary}{Corollary}

\begin{document}
\title{Decision-theoretic justifications for Bayesian hypothesis testing using credible sets}

\author
{ M{\aa}ns Thulin$^{1}$}
\date{}

\maketitle

\footnotetext[1]{Department of Mathematics, Uppsala University, Box 480, 751 06 Uppsala, Sweden.\\Phone: +46(0)184713389; E-mail: thulin@math.uu.se}

% Start doublespacing
%\doublespacing

\begin{abstract}
\noindent In Bayesian statistics the precise point-null hypothesis $\theta=\theta_0$ can be tested by checking whether $\theta_0$ is contained in a credible set. This permits testing of $\theta=\theta_0$ without having to put prior probabilities on the hypotheses. While such inversions of credible sets have a long history in Bayesian inference, they have been criticised for lacking decision-theoretic justification.

We argue that these tests have many advantages over the standard Bayesian tests that use point-mass probabilities on the null hypothesis. We present a decision-theoretic justification for the inversion of central credible intervals, and in a special case HPD sets, by studying a three-decision problem with directional conclusions. Interpreting the loss function used in the justification, we discuss when test based on credible sets are applicable.

We then give some justifications for using credible sets when testing composite hypotheses, showing that tests based on credible sets coincide with standard tests in this setting.

\end{abstract}

%%%%%%%%%%%%%%%%%%%%%%%%%%%%%%%%%%%%%%%%%%%%%%%%%%%%%%%%%%%%%%%%%%%%
\section{Introduction}\label{intro}
The first step of the standard solution to Bayesian testing of the point-null (or sharp, or precise) hypothesis $\theta\in\Theta_0=\{\theta_0\}$ is to assign a prior probability to the hypothesis. This is necessary if one wishes to use e.g. tests based on Bayes factors, as an absolutely continuous prior for $\theta$ would yield $\P(\theta\in\Theta_0|x)=0$ regardless of the data $x$. We shall refer to this procedure, described in detail e.g. by \citet[Section 5.2.4]{ro2}, as the standard test or the standard solution in the following. %These priors are technical constructs that can make Bayesian testing of point-null hypotheses seem artifical and cumbersome.
A continuous prior can however be utilized in a simpler alternative strategy, in which the evidence against $\Theta_0$ is evaluated indirectly by using inverted credible sets for testing. In this procedure a credible set $\Theta_c$, such that $\P(\theta\in\Theta_c|x)\geq 1-\alpha$, is computed and the null hypothesis is rejected if $\theta_0\notin\Theta_c$. Using credible sets for testing avoids the additional complication of imposing explicit probabilities on the hypotheses, avoids Lindley's paradox and allows for the use of non-informative priors.

While it has been argued that tests of point-null hypotheses are unnatural from a Bayesian point of view, they are undeniably of considerable practical interest (Berger \& Delampady, 1987; Ghosh et al., 2006, Section 2.7.2; Robert, 2007, Section 5.2.4). Inverting credible sets is arguably the simplest way to test such hypotheses. For this reason, inverting credible sets as a means to test hypotheses has long been, and continues to be, a part of Bayesian statistics. Tests of this type have been studied by \citet{bt1}, \citet{hsu1}, \citet{kim1}, \citet{dr1} and \cite{kr1}, among many others. \citet[Section 10.2]{ze1} credited this procedure to \citet[Section 5.6]{li1}, who in turn credited it to \citet{je1}. Lindley suggested that it may be useful when the prior information is vague or diffuse, with the caveat that the distribution of $\theta$ must be resonably smooth around $\theta_0$ for the inversion to be sensible. \citet[Section 3.4]{ko1} motivated this procedure by its resemblance to frequentist methods. \citet[Section 2.7]{gh1} described the inversion of credible sets as ``a very informal way of testing''. Zellner remarked that ``no decision theoretic justification appears available for this procedure''.

The purpose of this short note is to discuss such justifications and to put inverted credible sets on a decision-theoretic footing, making them formal Bayesian tests. These justifications shed light on when tests based on credible sets should be used. Throughout the text we assume that the parameter space $\Theta\subseteq \mathbb{R}$ and that the posterior distribution of $\theta$ is proper and absolutely continuous. %with respect to Lebesgue measure.

We will study three types of credible sets. A \emph{highest posterior density (HPD) set} contains all $\theta$ for which the posterior density is greater than some $k_\alpha$. Its appeal lies in the fact that among all $1-\alpha$ credible sets, the HPD set has the shortest length.

With $q_{\alpha}(\theta|x)$ denoting the posterior quantile function of $\theta$, such that $\P(\theta< q_{\alpha}(\theta|x))=\alpha$, the interval $(q_{\alpha/2}(\theta|x),q_{1-\alpha/2}(\theta|x))$ is called a $1-\alpha$ \emph{central interval} (also known as an equal-tailed, symmetric or probability centred interval). Central intervals are often used due to the fact that they are computationally simpler than HPD sets, especially when the HPD set is not connected. They are also more in line with frequentist practice.

A credible set is a \emph{credible bound} (or a one-sided credible interval) if it is of the form $\{\theta:\theta\leq q_{1-\alpha}(\theta|x)\}$ or $\{\theta:\theta\geq q_{\alpha}(\theta|x)\}$. Such sets are of interest when a lower or upper bound for the unknown parameter $\theta$ is desired.

First, we consider tests of the point-null hypothesis $\Theta_0=\{\theta_0\}$. We begin by comparing test based on credible sets compared to the standard solution with point-mass priors in Section \ref{criticism}, arguing that the former have many advantages. We then review a decision-theoretic justification of inverted HPD sets proposed by \citet{mew1} in Section \ref{pereira}. This justification involves a loss function that is non-standard in that it depends on the sample $x$. In Section \ref{central} we propose a justification of inverted central intervals by recasting the problem of testing $\Theta_0$ in a three-decision setting. This allows us to use a weighted 0--1 loss to justify tests based on central intervals, avoiding losses that depend on $x$. In Section \ref{which} we argue that test using central intervals are preferable to tests using HPD sets. We then turn our attention to tests of composite hypotheses in Section \ref{composite} and show that inverting credible sets lead to standard tests in this setting.

%%%%%%%%%%%%%%%%%%%%%%%%%%%%%%%%%%%%%%%%%%%%%%%%%%%%%%%%%%%%%%%%%%%%

\section{Testing point-null hypotheses}\label{pointnull}

\subsection{Contrasting the standard solution and credible sets}\label{criticism}
Consider tests of the point-null hypothesis $\Theta_0=\{\theta_0\}$ against $\Theta_1=\{\theta:\theta\neq\theta_0\}$. We will discuss two problems associated with the standard solution to point-null hypothesis testing; the use of mixture priors and its poor frequentist properties; that are overcome by using inverted credible sets instead. We then discuss the main drawback of tests based on credible sets, namely the lack of direct measures of evidence. In the below, $\pi_0>0$ denotes the prior probability of $\Theta_0$ in the standard solution, with $\pi_1=1-\pi_0$ being the prior probability of $\Theta_1$.

\emph{1. The use of mixture priors in the standard test.}
One could argue, as \citet[Section 2]{bes1}, that when using the standard solution it is in general reasonable to have $\pi_0\geq 1/2$ so that $\Theta_0$ is not rejected because of its low prior probability. This is common in practice. It is however much more sensible to adjust the loss function if stronger evidence is required before $\Theta_0$ is rejected. The prior distribution should reflect the prior beliefs (or lack of them) and not the severity of rejecting $\Theta_0$. Choosing a large $\pi_0$ corresponds to using a prior with a sharp spike close to $\theta_0$. Such priors are only reasonable if there is a very strong prior belief in the null hypothesis, and cannot be used in an analysis where at least some degree of objectivity is desired.

When lacking prior information or striving for complete objectivity, it is common practice to use a prior that in some sense is non-informative or objective. The point-null hypothesis is often a simplification of the hypothesis that $|\theta-\theta_0|<\epsilon$ for some small $\epsilon$, in which case $\pi_0=\P(|\theta-\theta_0|<\epsilon)$. If $\P(|\theta-\theta_0|<\epsilon)$ is computed under an objective prior on $\theta$, $\pi_0$ will invariably be extremely small. The ``objective'' prior $\pi_0=1/2$ is in fact heavily biased towards $\Theta_0$.

\citet[Section 2]{bes1} claim that a point-null test only makes sense to a Bayesian if the prior actually has a sharp spike near $\theta_0$. But it is perfectly reasonable to formulate a hypothesis before constructing an informative prior or to require that an objective prior is used when no prior information is available. It is furthermore reasonable to require that the prior distribution can be used for more than one type of statistical decision, e.g. both estimation and hypothesis testing; the prior should reflect prior beliefs and not the type of decision that one is interested in.

\emph{2. The poor frequentist performance of the standard test.} There are many situations in which it is reasonable to require that a statistical procedure works well from both a Bayesian and a frequentist perspective, particularly when an objective analysis is desirable. A well-known complication with the standard solution to testing point-null hypotheses is the asymptotic discrepancy between Bayesian and frequentist analysis that is known as Lindley's paradox \citep{li2}, in which, for any fixed prior, $\P(\theta\in\Theta_0|x)$ goes to 1 as the sample size increases for values of a test statistic that correspond to a fixed (small) $p$-value. Thus a frequentist and a Bayesian will reach different conclusions as more and more data is collected. This is another example of how the standard Bayesian test favours the null hypothesis.

There is no such discrepancy for tests based on credible sets. On the contrary, when based on non-informative priors, credible sets tend to have favourable properties when treated as frequentist confidence intervals; see e.g. \citet{bcd1}; and hypothesis testing using a credible set is often a valid frequentist procedure.

\emph{3. Measures of evidence.}
The main argument against the use of credible sets for testing is that they do not utilize $\P(\theta\in\Theta_0|x)$, and so does not really measure the amount of evidence for and against $\Theta_0$. Alas, the first part of this argument seems weak, considering the contradictions and paradoxes that the standard tests utilizing $\P(\theta\in\Theta_0|x)$ give rise to. It should also be noted that a test can be fully Bayesian without being based on $\P(\theta\in\Theta_0|x)$ as long at it depends only on the posterior distribution. Tests based on inverted credible sets have this property.

While credible sets do not measure the evidence against $\Theta_0$ directly, there are indirect measures associated with such tests. With $T(x)$ being the largest HPD set not containing $\theta_0$, \citet{ps1} proposed $\P(\theta\notin T(x)|x)$, i.e. the smallest $\alpha$ for which $\theta_0$ is not contained in the $1-\alpha$ HPD set, as a measure of evidence against $\Theta_0$, with values close to $0$ indicating that $\Theta_0$ is false. Similar in spirit to the frequentist $p$-value, this is a measure of how far out in the tails of the posterior distribution $\theta_0$ is. With $S(x)$ being the largest central interval not containing $\theta_0$, one can similarly use $\P(\theta\notin S(x)|x)$ as a measure of evidence. When $T(x)$ and $S(x)$ are viewed as frequentist confidence intervals, these measures of evidence coincide with the $p$-values of the corresponding two-sided tests. This stands in sharp contrast to the evidence measure $\P(\theta\in\Theta_0|x)$ utilized in the standard solution, which is irreconcilable with $p$-values \citep{bes1}.

%%%%%%%%%%%%%%%%%%%%%%%%%%%%%%%%%%%%%%%%%%%%%%%%%%%%%%%%%%%%%%%%%%%%

\subsection{The Madruga--Esteves--Wechsler loss and HPD sets}\label{pereira}
\citet{ps1} proposed, along with the $p$-value-like measure of evidence mentioned above, a test of $\Theta_0=\{\theta_0\}$ that is equivalent to inverting HPD sets, dubbing it the Full Bayesian Significance Test. See also \citet{mps1} and \citet{psw1}. None of the aforementioned papers discuss the historical background of this test, briefly outlined in Section \ref{intro} of the present paper.

Let $\pi(\cdot)$ denote the \emph{a priori} density function of $\theta$ and define $T(x)=\{ \theta:\pi(\theta|x)>\pi(\theta_0|x)\}.$ $T(x)$ is the most credible HPD set not containing $\theta_0$. The Pereira--Stern test rejects $\Theta_0$ when $\P(\theta\notin T(x)|x)$ is "small" ($<0.05$, say). In other words, $\Theta_0$ is rejected at the $5~\%$ level if and only if it is not contained in the $95~\%$ HPD set.

Madruga, Esteves and Wechsler (2001) provided a decision-theoretic justification for the Pereira-Stern test, that is non-standard in that the loss function depends on $x$. Let the test function $\varphi$ be $0$ if $\Theta_0$ is accepted and $1$ if $\Theta_0$ is rejected. The Madruga-Esteves-Wechsler loss is:
$$
L_{MEW}(\theta,\varphi,x) = \begin{cases} a(1-\mathbb{I}(\theta\in T(x)), & \mbox{if } \varphi(x)=1 \\
b+c\mathbb{I}(\theta\in T(x)), & \mbox{if } \varphi(x)=0, \end{cases}
$$
with $a,b,c>0$. Minimization of the expected posterior loss leads to the Pereira--Stern test, where $\Theta_0$ is rejected if $\P(\theta\notin T(x)|x)<(b+c)/(a+c)$.
As an example, the test resulting from the choice $a=0.975$ and $b=c=0.025$ is equivalent to inverting the 95 \% HPD set.

The controversial part of this decision-theoretic justification of inverting HPD sets is that the loss function depends on $x$. While such loss functions have appeared in the literature a few times; \citet{mew1} give \citet{ka1} and \citet[Section 6.1.4]{bs1} as examples; they do not appear to be widely accepted. Indeed, choosing the loss function after $x$ has been observed is arguably not entirely in line with classic decision theory. \citet{psw1} argue that an advantage with this type of loss is that it allows the statistician to include his or her embarrassment over accepting $\Theta_0$ when $\theta_0$ is not in a particular HPD set. Such loss functions seems applicable in highly subjective analyses, but should be avoided when objectivity is desirable. As we will show next, tests based on central intervals can, in contrast, be justified using a standard loss functions.

%%%%%%%%%%%%%%%%%%%%%%%%%%%%%%%%%%%%%%%%%%%%%%%%%%%%%%%%%%%%%%%%%%%%%5

\subsection{Central intervals}\label{central}
 When testing $\Theta_0$, credible (and confidence) intervals are in practice often used for directional conclusions. If $\theta_0$ is above (or below) the credible interval, it is common practice to reject $\Theta_0$ and to conclude that $\theta$ is larger (or smaller) than $\theta_0$.

Next, we provide a decision-theoretic justification of the inversion of central intervals that uses a standard loss function not involving $x$, and for which the directional conclusions are fully valid. A justification for the use of HPD sets is obtained in a special case as a corollary.

We formally reformulate the problem of testing the point-null hypothesis $\Theta_0=\{\theta_0\}$ as a three-decision problem with directional conclusions. $\Theta_0$ is then tested against both $\Theta_{-1}=\{\theta:\theta<\theta_0\}$ and $\Theta_{1}=\{\theta:\theta>\theta_0\}$. Such problems have previously been studied by for instance \citet{jt1} and \citet{jo1} in the frequentist setting and \citet{ba2} and \citet{ba1} in the Bayesian setting.

\begin{theorem}\label{centthm}
Let $\Theta_0=\{\theta_0\}$, $\Theta_{-1}=\{\theta:\theta<\theta_0\}$ and $\Theta_{1}=\{\theta:\theta>\theta_0\}$. Let $\varphi$ be a decision function that takes values in $\{-1,0,1\}$, such that we accept $\Theta_i$ if $\varphi(x)=i$. Under an absolutely continuous prior for $\theta$ and the loss function
\[
L_{(1)}(\theta,\varphi) = \begin{cases} 0, & \mbox{if } \theta\in\Theta_i\mbox{ and }\varphi=i, \quad i\in \{-1,0,1\}, \\
\alpha/2, & \mbox{if } \theta\notin\Theta_0 \mbox{ and }\varphi=0,\\
1, & \mbox{if } \theta\in\Theta_{i}\cup\Theta_0 \mbox{ and }\varphi=-i,\quad i\in\{-1,1\},\end{cases}
\]
with $0<\alpha<1$, the Bayes test is to reject $\Theta_0$ if $\theta_0$ is not contained in the central $1-\alpha$ credible set.
\end{theorem}

\begin{proof}
The expected posterior loss is
\[\begin{split}
{\rm E}(L_{(1)}(\theta,\varphi)|x)&={\rm E}\Big(\P(\theta\in\Theta_1|x)\mathbb{I}_{\{-1\}}(\varphi)+\P(\theta\in\Theta_{-1}|x)\mathbb{I}_{\{1\}}(\varphi)\\&\qquad\qquad\qquad\qquad\qquad+\alpha/2\cdot\P(\theta\notin\Theta_0|x)\mathbb{I}_{\{0\}}(\varphi)\Big|x\Big)\\
&={\rm E}\Big(\P(\theta\in\Theta_1|x)\mathbb{I}_{\{-1\}}(\varphi)+\P(\theta\in\Theta_{-1}|x)\mathbb{I}_{\{1\}}(\varphi)+\alpha/2\cdot\mathbb{I}_{\{0\}}(\varphi)|x\Big),
\end{split}\]
since $\P(\theta\notin\Theta_0|x)=1$ because of the absolute continuity of $\theta$. The posterior expected losses are therefore ${\rm E}(L_{(1)}(\theta,0)|x)=\alpha/2$ and ${\rm E}(L_{(1)}(\theta,i|)x)=\P(\theta\in\Theta_i|x)$ for $i\in\{-1,1\}$. It follows that the Bayes decision rule is to accept $\Theta_0$ if
$\alpha/2<\min_{i\in\{-1,1\}}\P(\theta\in\Theta_i|x)$,
or, equivalently, if
$1-\alpha/2>\max_{i\in\{-1,1\}}\P(\theta\in\Theta_i|x)$.
But $\P(\theta\in\Theta_{-1}|x)>1-\alpha/2$ if and only if the upper bound $q_{1-\alpha/2}(\theta|x)<\theta_0$, in which case $\theta_0$ is not contained in the central interval. Similarly, $\P(\theta\in\Theta_{1}|x)>1-\alpha/2$ if and only if the lower bound $q_{\alpha/2}(\theta|x)>\theta_0$. Thus $\Theta_0$ is accepted if and only if $\theta$ is contained in the $1-\alpha$ central interval.
\end{proof}
It is not uncommon that the posterior density is unimodal and symmetric, typical cases of this being normal and Student's $t$ posteriors. In this case, the central interval coincides with the HPD set. The following corollary is therefore immediate.
\begin{corollary}
If $\pi(\theta|x)$ is absolutely continuous, symmetric and unimodal, the Bayes test under $L_{(1)}$ is to reject $\Theta_0$ if $\theta_0$ is not contained in the $1-\alpha$ HPD set.
\end{corollary}

Typically $\alpha\leq 0.1$ in $L_{(1)}$, in which case the evidence for either $\Theta_{-1}$ or $\Theta_1$ has to be quite large before $\Theta_0$ is rejected. This is a conservative type of analys\-is, where the statistician is somewhat reluctant to reject $\Theta_0$. Examples of situations where such a loss is applicable include many investigations of hypotheses in science, circumstances where falsely rejecting $\Theta_0$ leads to considerable economical losses and legal and forensic questions (this loss is in line with the classic \emph{in dubio pro reo} principle).

The posterior probability of the null hypothesis is not evaluated in this test (it is always 0), but the probabilities of the two alternative hypotheses are. By the process of elimination, $\Theta_0$ is accepted if neither $\Theta_{-1}$ nor $\Theta_1$ has a high enough posterior probability. $\Theta_0$ is rejected if $\theta_0$ is far out in the tails of the posterior distribution, making these tests quite similar to frequentist hypothesis testing.

%%%%%%%%%%%%%%%%%%%%%%%%%%%%%%%%%%%%%%%%%%%%%%%%%%%%%%%%%%%%%%%%%%%%%5

\subsection{Reasons to use central intervals rather than HPD sets for testing}\label{which}

For historical reasons, test based on HPD sets are much more common in the literature than tests using central intervals; confer the references given in Section \ref{intro} for examples. However, from a decision-theoretic perspective, the somewhat tautological use of the HPD set $T(x)$ in the loss function $L_{MEW}$ seems artifical in comparison to the easy-to-interpret loss function $L_{(1)}$. The $L_{MEW}$--justification of tests using HPD sets involves a good deal more subjectivity than does the justification of tests using central intervals. In most situations, the latter test seems to be preferable when a test with a decision-theoretic justification is desirable.

A further argument for prefering inverted central intervals is given by studying \emph{how} the two types of tests rejects the null hypothesis.
\begin{figure}
\begin{center}
 \caption{Power of two-sided tests of the hypothesis $H_0: \sigma^2=1$, where $\sigma^2$ is the variance of the normal distribution with known mean, based on 95 \% credible sets using the Jeffreys prior.}\label{fig1}
   \includegraphics[width=\textwidth]{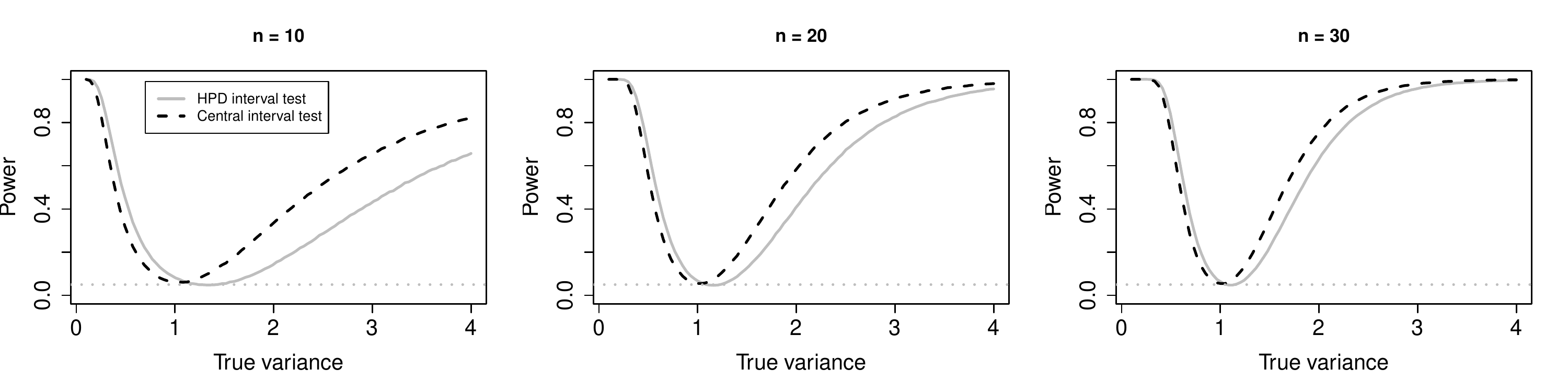}
 \end{center}
\end{figure}
For $\theta\in\mathbb{R}$, when $\Theta_0=\{\theta_0\}$ is rejected, there is always an implicit directional aspect to this decision: the null hypothesis is rejected because it seems more likely that $\theta$ belongs to either $\Theta_{-1}=\{\theta:\theta<\theta_0\}$ or $\Theta_{1}=\{\theta:\theta>\theta_0\}$. When the posterior is skewed, HPD sets are, by construction, biased towards one of these sub-alternatives. Such directional rejection biases can be avoided by using inverted central intervals for testing. 

When the frequentist properties of test based on HPD and central intervals are evaluated in some common point-null testing problems for scale or rate parameters, it is seen that this rejection bias causes the HPD set test to have worse power properties than the central interval test. For inference about the scale/rate parameters in the normal, gamma, inverse gamma and Weibull distributions using the corresponding Jeffreys priors, the two tests have comparable power for $\theta<\theta_0$, while the test using the central interval has higher power for $\theta>\theta_0$; sometimes substantially so. Examples are given in Figures \ref{fig1} and \ref{fig2}; the power functions for the other cases are similar. In general, the central interval test also seems to come closer to attaining the nominal size in this setting.

A similar issue occurs when the posterior density is monotone, as is for instance the case when using the conjugate Pareto prior for $\theta$ in $U(0,\theta)$. In this setting the HPD set is an upper or lower confidence bound, meaning that it only is possible to reject $\Theta_0$ in one direction, so that the resulting test in fact is one-sided.

\begin{figure}
\begin{center}
 \caption{Power of two-sided tests of the hypothesis $H_0: \lambda=1$, where $\lambda$ is the rate parameter of the exponential distribution, based on 95 \% credible sets  using the Jeffreys prior.}\label{fig2}
   \includegraphics[width=\textwidth]{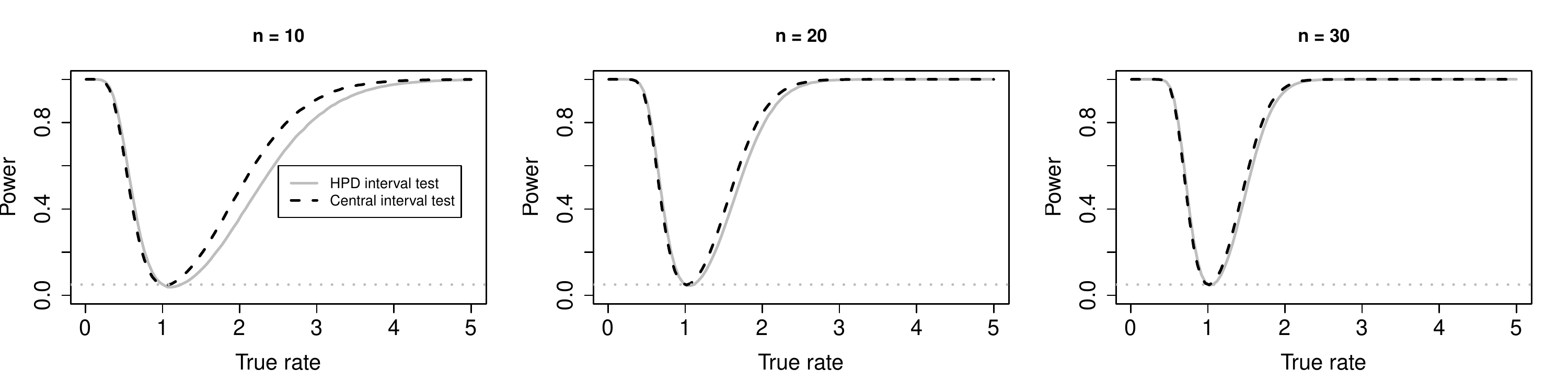}  
\end{center}
\end{figure}

%%%%%%%%%%%%%%%%%%%%%%%%%%%%%%%%%%%%%%%%%%%%%%%%%%%%%%%%%%%%%%%%%%%%%5

%%%%%%%%%%%%%%%%%%%%%%%%%%%%%%%%%%%%%%%%%%%%%%%%%%%%%%%%%%%%%%%%%%%%

\section{Testing composite hypotheses}\label{composite}
When performing tests of composite hypotheses, there is typically no discrepancy between tests based on credible set and the standard Bayesian tests. This will be illustrated with two examples below. For $\Theta_0$ and $\Theta_1$ being uncountable subsets of $\mathbb{R}$, both with positive probabilities under an absolutely continuous prior, we will consider the weighted 0--1 loss function

\[
L_{(2)}(\theta,\varphi) = \begin{cases} 0, & \mbox{if } \varphi=1-\mathbb{I}_{\Theta_0}(\theta) \\
a, & \mbox{if } \theta\in\Theta_0 \mbox{ and }\varphi=1\\
b, & \mbox{if } \theta\in\Theta_1 \mbox{ and }\varphi=0. \end{cases}
\]

First, we consider one-sided composite hypotheses. The proof of the following theorem is in analogue with that of Theorem \ref{centthm} and is therefore omitted.

\begin{theorem}\label{thm3}
For the hypotheses $\Theta_0=\{\theta:\theta\leq \theta_0\}$ and $\Theta_1=\{\theta:\theta>\theta_0\}$, with $\P(\Theta_0), \P(\Theta_1)>0$, let $\varphi$ be a test function, such that $\Theta_0$ is rejected if $\varphi=1$ and accepted if $\varphi=0$. Under an absolutely continuous prior for $\theta$ and the loss function $L_{(2)}$ with $a=1-\alpha$ and $b=\alpha$, $0<\alpha<1$, the Bayes test is to reject $\Theta_0$ if and only if $\theta_0$ is not contained in the lower-bound credible set $\{\theta:\theta\geq q_{\alpha}(\theta|x)\}$, or, equivalently, if and only if $\P(\theta\in\Theta_0|x)\leq\alpha$.
\end{theorem}

%\begin{proof}
%The Bayes test is to accept $\Theta_0$ if $$\P(\theta\in\Theta_0|x)\geq \alpha(1-\alpha+\alpha)^{-1}=\alpha,$$
%see e.g. Proposition 5.2.2  of \citet{ro2}. But $\P(\theta\in\Theta_0|x)\geq\alpha$ if and only if the lower-bound $1-\alpha$ credible set $\{\theta:\theta\geq q_{\alpha}(\theta|x)\}$ contains $\theta_0$. This test is therefore equivalent to checking whether $\theta_0$ is in the lower-bound credible set.
%\end{proof}

This is a standard one-sided test and the credible set test is thus in agreement with the standard methods. The corresponding justification for upper-bound credible sets is given by considering testing the hypothesis $\Theta_0=\{\theta:\theta\geq \theta_0\}$ against $\Theta_1=\{\theta:\theta<\theta_0\}$ instead.

For a general composite hypothesis $\theta\in\Theta_0$ with the alternative $\theta\in\Theta_1=\Theta\backslash\Theta_0$, a correspondence between hypothesis testing and credible test can be established using a loss function where falsely accepting the null hypothesis is considered to be much worse than falsely rejecting it. Such a loss can be of considerable use in situations where false acceptance of the null hypothesis can be costly, for instance when screening for diseases, malicious web pages or signs of terrorist activity.

\begin{theorem}\label{allthm}
For the hypotheses $\Theta_0$ and $\Theta_1=\Theta\backslash\Theta_0$, with $\P(\Theta_0), \P(\Theta_1)>0$, let $\varphi$ be a test function, such that $\Theta_0$ is rejected if $\varphi=1$ and accepted if $\varphi=0$. Under an absolutely continuous prior for $\theta$ and the loss function $L_{(2)}$ with $a=\alpha$ and $b=1-\alpha$, $0<\alpha<1/2$, the Bayes test is to reject $\Theta_0$ if and only if there exists at least one $\alpha$ credible set that does not contain a non-null subset of $\Theta_0$, or, equivalently, if and only if $\P(\theta\in\Theta_0|x)\leq1-\alpha$.
\end{theorem}

\begin{proof}
By Proposition 5.2.2  of \citet{ro2}, the Bayes test is to accept $\Theta_0$ if $\P(\theta\in\Theta_0|x)\geq 1-\alpha$. Now, a credible set $\Theta_c$ is a set such that $\P(\Theta_c|x)\geq 1-\alpha$. Thus, by definition, if $\Theta_0$ is such that $\P(\Theta_0|x)\geq 1-\alpha$, $\Theta_c$ can be a credible set only if $\P(\Theta_0\cap\Theta_c|x)>0$. Hence, we accept the null hypothesis if and only if \emph{each} $1-\alpha$ credible set contains a non-null subset of $\Theta_0$.
\end{proof}
Once again, this is a standard test, albeit a less commonly used one. Using this loss function corresponds to the statistician being either unsure about which credible set to report or unwilling to accept $\Theta_0$ if there exists at least one way to look at the data that puts $\Theta_0$ to question.%\\[2mm]

%%%%%%%%%%%%%%%%%%%%%%%%%%%%%%%%%%%%%%%%%%%%%%%%%%%%%%%%%%%%%%%%%%%%

\subsection*{Acknowledgments}
The author wishes to thank Silvelyn Zwanzig for valuable suggestions that helped improve the exposition of the paper, Christian Robert for pointing out a crucial mistake in an earlier draft and Paulo Marques, who helped provide some key references in an early stage of this work.

%%%%%%%%%%%%%%%%%%%%%%%%%%%%%%%%%%%%%%%%%%%%%%%%%%%%%%%%%%%%%%%%%%%%%%%%%%%%%%%%%%
%\pagebreak


\begin{thebibliography}{99}
%\footnotesize{

\bibitem[\protect\astroncite{Bansal et al.}{2012}]{ba1} Bansal, N.K., Hamedani, G.G., Sheng, R. (2012), Bayesian analysis of hypothesis testing problems for general population: a Kullback–Leibler alternative, \emph{Journal of Statistical Planning and Inference}, \textbf{142}, 1991--1998.

\bibitem[\protect\astroncite{Bansal \& Sheng}{2010}]{ba2} Bansal, N.K., Sheng, R. (2010), Bayesian decision theoretic approach to hypothesis testing problems with skewed alternatives, \emph{Journal of Statistical Planning and Inference}, \textbf{140}, 2894--2903.

\bibitem[\protect\astroncite{Berger \& Delampady}{1987}]{bd1} Berger, J.O., Delampady, M. (1987), Testing precise hypotheses, \emph{Statistical Science}, \textbf{2}, 317--335.

\bibitem[\protect\astroncite{Berger \& Sellke}{1987}]{bes1} Berger, J.O., Sellke, T. (1987), Testing a point null hypothesis: the irreconcilability of p values and evidence, \emph{Journal of the American Statistical Association}, \textbf{82}, 112--122.

%\bibitem[\protect\astroncite{Bernardo}{2005}]{be1} Bernardo, J.M. (2005), Intrinsic credible regions: an objective Bayesian approach to interval estimation, \emph{Test}, \textbf{14}, 317-384.

%\bibitem[\protect\astroncite{Bernardo \& Rueda}{2002}]{br1} Bernardo, J.M., Rueda, R. (2002), Bayesian hypothesis testing: a reference approach, \emph{International Statistical Review}, \textbf{70}, 351-372.

\bibitem[\protect\astroncite{Bernardo \& Smith}{1994}]{bs1} Bernardo, J.M., Smith, A.F.M. (1994), \emph{Bayesian Theory}, John Wiley \& Sons, New York.

\bibitem[\protect\astroncite{Box \& Tiao}{1965}]{bt1}  Box, G.E.P., Tiao, G.C. (1965), Multiparameter problems from a Bayesian point of view, \emph{The Annals of Mathematical Statistics}, \textbf{36}, 1468--1482.

\bibitem[\protect\astroncite{Brown et al.}{2001}]{bcd1} Brown, L.D., Cai, T.T., DasGupta, A. (2001), {Interval estimation for a binomial proportion}, \emph{Statistical Science}, \textbf{16}, 101--133.

%\bibitem[\protect\astroncite{Casella \& Berger}{1987}]{cb1} Casella, G., Berger, R.L. (1987), Reconciling Bayesian and frequentist evidence in the one-sided testing problem, \emph{Journal of the American Statistical Association}, \textbf{82}, 106-111.

\bibitem[\protect\astroncite{Drummond \& Rambaut}{2007}]{dr1}  Drummond, A.J., Rambaut, A. (2007), BEAST: Bayesian evolutionary analysis by sampling trees, \emph{BMC Evolutionary Biology}, \textbf{7}, 214.

\bibitem[\protect\astroncite{Ghosh et al.}{2006}]{gh1} Ghosh, J.K., Delampady, M., Samanta, T. (2006), \emph{An Introduction to Bayesian Analysis}, Springer, New York.

%\bibitem[\protect\astroncite{Gir\'{o}n \& Moreno}{2005}]{be2} Gir\'{o}n, J., Moreno, E. (2005), Comment to Intrinsic credible regions: an objective Bayesian approach to interval estimation, \emph{Test}, \textbf{14}, 362-364.

\bibitem[\protect\astroncite{Hsu}{1982}]{hsu1} Hsu, D.A. (1982), Bayesian robust detection of shift in the risk structure of stock returns, \emph{Journal of the American Statistical Association}, \textbf{77}, 29--39.

\bibitem[\protect\astroncite{Jeffreys}{1961}]{je1} Jeffreys, H. (1961), \emph{Theory of Probability},  3rd edition,  Clarendon Press, Oxford.

\bibitem[\protect\astroncite{Jones \& Tukey}{2000}]{jt1}  Jones, L.V., Tukey, J.W. (2000), A sensible formulation of the significance test, \emph{Psychological Methods}, \textbf{5}, 411--414.

\bibitem[\protect\astroncite{Jonsson}{2012}]{jo1}  Jonsson, F. (2012), Characterizing optimality among three-decision procedures for directional conclusions, \emph{Journal of Statistical Planning and Inference}, \textbf{143}, 392--399.

\bibitem[\protect\astroncite{Kadane}{1992}]{ka1}  Kadane, J.B. (1992), Healthy scepticism as an expected utility explanation of the phenomena of Allais and Ellsberg, \emph{Theory and Decision}, \textbf{32}, 57--64.

\bibitem[\protect\astroncite{Kim}{1991}]{kim1} Kim, D. (1991), A Bayesian significance test of the stationarity of regression parameters, \emph{Biometrika}, \textbf{78}, 667--675.

\bibitem[\protect\astroncite{Koch}{2007}]{ko1} Koch, K.-R. (2007), \emph{Introduction to Bayesian Statistics}, Springer, Berlin.

\bibitem[\protect\astroncite{Kruschke}{2011}]{kr1} Kruschke, J.K. (2011), Bayesian assessment of null values via parameter estimation and model comparison, \emph{Perspectives on Psychological Science}, \textbf{6}, 299--312.

\bibitem[\protect\astroncite{Lindley}{1957}]{li2} Lindley, D.V. (1957), A statistical paradox, \emph{Biometrika}, \textbf{44}, 187--192.

\bibitem[\protect\astroncite{Lindley}{1965}]{li1} Lindley, D.V. (1965), \emph{Introduction to Probability and Statistics from a Bayesian Viewpoint, Part 2. Inference},  Cambridge University Press, Cambridge.

\bibitem[\protect\astroncite{Madruga et al.}{2001}]{mew1}  Madruga M.R., Esteves, L.G., Wechsler, S. (2001), On the Bayesianity of Pereira-Stern tests, \emph{Test}, \textbf{10}, 291--299.

\bibitem[\protect\astroncite{Madruga et al.}{2003}]{mps1}  Madruga, M.R.,  Pereira, C.A.B., Stern, J.M. (2003), Bayesian evidence test for precise hypotheses, \emph{Journal of Statistical Planning and Inference}, \textbf{117}, 185--198.

\bibitem[\protect\astroncite{Pereira \& Stern}{1999}]{ps1} Pereira, C.A.B., Stern, J.M. (1999), Evidence and credibility: full Bayesian significance test for precise hypotheses, \emph{Entropy}, \textbf{1}, 99--110.

\bibitem[\protect\astroncite{Pereira et al.}{2008}]{psw1}  Pereira, C.A.B., Stern, J.M., Wechsler, S. (2008), Can a significance test be genuinely Bayesian?, \emph{Bayesian Analysis}, \textbf{3}, 79--100.

%\bibitem[\protect\astroncite{Robert}{1996}]{ro1} Robert, C.P. (1996), Intrinsic loss functions, \emph{Theory and Decision}, \textbf{40}, 192-214.

\bibitem[\protect\astroncite{Robert}{2007}]{ro2} Robert, C.P. (2007), \emph{The Bayesian Choice}, Springer, New York.

\bibitem[\protect\astroncite{Zellner}{1971}]{ze1} Zellner, A. (1971), \emph{An Introduction to Bayesian Inference in Econometrics}, John Wiley \& Sons, New York.

%}
\end{thebibliography}
\end{document}